\documentclass[preprint,12pt]{elsarticle}

\usepackage{amsmath}
\usepackage{amssymb}
\usepackage{amsthm}
\usepackage{mathrsfs}
\usepackage{latexsym}
\usepackage{mathtools}
\usepackage{enumerate}
\usepackage{multirow}
\usepackage{multicol}
\usepackage[all]{xy}

\usepackage[bookmarks]{hyperref}

\numberwithin{equation}{section}
\usepackage{url}
\usepackage{graphicx}
\usepackage{color}

\newtheorem{thm}{Theorem}
\newtheorem{lem}{Lemma}[section]
\newtheorem{prop}[lem]{Proposition}
\newtheorem{cor}[lem]{Corollary}

\newtheorem{defn}[lem]{Definition}

\newdefinition{rem}[lem]{Remark}
\newdefinition{ex}[lem]{Example}
\newproof{pot1}{Proof of Theorem \ref{Thm1}}
\newproof{pot2}{Proof of Theorem \ref{Thm2}}
\newproof{pot3}{Proof of Theorem \ref{Thm3}}

\journal{Journal of Algebra}

\newcommand{\Au}{\mathrm{Aut}}
\newcommand{\ad}{\mathrm{ad}}
\newcommand{\bd}{\bar{d}}
\newcommand{\bl}{\bar{L}}

\newcommand{\co}{\coloneqq}
\newcommand{\D}{\mathrm{Der}}
\newcommand{\di}{\mathrm{diag}}
\newcommand{\F}{\mathbb{F}}
\newcommand{\g}{\mathfrak{g}}
\newcommand{\GL}{\mathrm{GL}}
\newcommand{\h}{\mathfrak{h}}
\newcommand{\hk}{H^1(K, K)}
\newcommand{\I}{\mathrm{Im} \,}
\newcommand{\Lnone}{$\mathrm{Lie}(n+1,n)$}
\newcommand{\Lntwo}{$\mathrm{Lie}(n+2,n)$}
\newcommand{\Lntwoa}{$\mathrm{Lie}_\ad(n+2,n)$}
\newcommand{\Lntwoc}{$\mathrm{Lie}_c(n+2,n)$}
\newcommand{\Li}{\mathrm{Lie}(n+1,n)}
\newcommand{\li}{\mathrm{Lie}(n+2,n)}
\newcommand{\lia}{\mathrm{Lie}_\ad(n+2,n)}
\newcommand{\lic}{\mathrm{Lie}_c(n+2,n)}

\newcommand{\M}{\mathrm{Mat}}
\newcommand{\R}{\mathbb{R}}
\newcommand{\ra}{\mathrm{rank}\,}
\newcommand{\s}{\mathrm{span}}
\newcommand{\td}{\tilde{d}}
\newcommand{\ts}{\tilde{\sigma}}
\newcommand{\Z}{\mathcal{Z}}

\begin{document}

\begin{frontmatter}



\title{On the problem of classifying solvable Lie algebras having small codimensional derived algebras}


\author[1]{Hoa Q. Duong}
\ead{hoa.duongquang@hoasen.edu.vn}
\address[1]{Faculty of Information Technology, Hoa Sen University, Ho Chi Minh City, Vietnam}

\author[2]{Vu A. Le}
\ead{vula@uel.edu.vn}
\address[2]{Department of Economic Mathematics, University of Economics and Law, 
Vietnam National University -- Ho Chi Minh City, Vietnam}

\author[3]{Tuan A. Nguyen}
\ead{natuan@upes.edu.vn}
\address[3]{Faculty of Political Science and Pedagogy, Ho Chi Minh City University of Physical Education and Sport, Vietnam}

\author[4]{Hai T. T. Cao}
\ead{tuhai.thptlequydon@ninhthuan.edu.vn}
\address[4]{Department of Mathematics and Informatics, Ho Chi Minh City University of Education, Vietnam}

\author[5]{Thieu N. Vo\corref{cor1}}
\ead{vongocthieu@tdtu.edu.vn}
\address[5]{Fractional Calculus, Optimization and Algebra Research Group, Faculty of Mathematics and Statistics, 
Ton Duc Thang University, Ho Chi Minh City, Vietnam}
\cortext[cor1]{Corresponding author.}

\begin{abstract}
This paper concerns the problem of classifying finite-dimensional real solvable Lie algebras 
whose derived algebras are of codimension 1 or 2. On the one hand, we present an effective 
method to classify all $(n+1)$-dimensional real solvable Lie algebras having 1-codimensional 
derived algebras provided that a full classification of $n$-dimensional nilpotent Lie algebras is given.
On the other hand, the problem of classifying all $(n+2)$-dimensional real solvable Lie algebras 
having 2-codimensional derived algebras is proved to be wild. In this case, we provide a method 
to classify a subclass of the considered Lie algebras which are extended from their derived algebras 
by a pair of derivations containing at least one inner derivation.
\end{abstract}


\end{frontmatter}


\section{Introduction}\label{sec1}

A classification problem is called to be \emph{wild} if it contains the problem of classifying pairs of matrices up to 
simultaneous similarity (see \cite{DF72,DF73}).
According to Belitskii and Sergeichuk \cite[Section 1]{BS03}, wild problems are 
hopeless in a certain sense. Several classification problems were pointed out to be wild 
(see \cite{BDLST09,BLS05,BP19,BS03,FKPS18,Ser00} and references therein).

Unfortunately, the problem of classifying solvable Lie algebras is wild. Indeed, Belitskii et al. \cite[Theorem 4]{BDLST09} 
proved that the problem of classifying 2-step nilpotent Lie algebras 
(over an algebraically closed field of characteristic other than two) with 3-dimensional derived algebras is wild. 
Then so is the problem of classifying all nilpotent Lie algebras.
Since the problem of classifying solvable Lie algebras contains 
the problem of classifying nilpotent Lie algebras, the former problem is wild, too.
As a consequence, the problem of classifying solvable Lie algebras is very difficult.
Although several partial results were obtained in low dimensions (see \cite{SW14}), a complete classification of solvable Lie algebras does not exist so far.

Recently, the problem of classifying solvable Lie algebras with a given derived algebra has been extensively studied. 
Real solvable Lie algebras with $1$-dimensional derived algebras are completely classified by Sch{\" o}bel \cite{Sch93}.
Partial results on classifications of solvable Lie algebras with $2$-dimensional derived algebras was obtained in \cite{Ebe03, Jan10, Sch93}.
Schur \cite{Sch05} and Jacobson \cite{Jac44} investigated formulas for determining the maximal dimension of a commutative subalgebra of a matrix Lie algebra.
Based on the results of \cite{Jac44, Sch05}, a full classification for real solvable Lie algebras with $2$-dimensional derived algebras was achieved in \cite{VTTTT19}.
To the best of our knowledge, the problem of classifying real solvable Lie algebras with the derived algebras of dimension $\neq 1,2$ still remains open.

This paper aims to study the classification problems for real solvable Lie algebras with high dimensional derived algebras.
We denote by {\Lnone} (resp., {\Lntwo}) 
the class of $(n+1)$-dimensional (resp., $(n+2)$-dimensional) real solvable Lie algebras whose derived algebras are of dimension $n$. 
Three main theorems of the paper are as follows.

First of all, for a given $n$-dimensional real nilpotent Lie algebra $K$, each Lie algebra in 
{\Lnone} admitting $K$ as the derived algebra is an extension of $K$ by a derivation of $K$. 
We point out that the derivation of this extension must be an outer derivation.
However, a Lie algebra extended from $K$ by an outer derivation of $K$ is not necessary in {\Lnone}.
We give necessary and sufficient conditions for the derivation so that the extension is in {\Lnone} (see Proposition~\ref{Pro3.1}).
Furthermore, we prove that:
\begin{thm}\label{Thm1}
	For an arbitrary $n$-dimensional nilpotent Lie algebra $K$, the problem of classifying all Lie algebras 
	in {\Lnone} with the derived algebra $K$ is equivalent to the problem of classifying 
	outer derivations in the first cohomology space $H^1(K,K)$ satisfying equivalent conditions in Proposition~\ref{Pro3.1}, up to proportional similarity.
\end{thm}

Similarly, for a given $n$-dimensional real nilpotent Lie algebra $H$, each Lie algebra in {\Lntwo} 
admitting $H$ as its derived algebra is an extension of $H$ by a pair of derivations. 
However, not every extension of $H$ by a pair of derivations is in {\Lntwo}. We give necessary and sufficient 
conditions for the pair of derivations so that the extension is always in {\Lntwo} (see Proposition~\ref{Pro4.1}). 
Based on the conditions, we prove that:
\begin{thm}\label{Thm2}
	The problem of classifying {\Lntwo} is wild.
\end{thm}

The wildness of the problem of classifying {\Lntwo} motivates us to consider special cases 
(see Belitskii et al. \cite[Section 3]{BLS05}). 
Based on the proof of Theorem~\ref{Thm2}, 
we consider the subclass \Lntwoa\ $\subseteq$ \Lntwo\  
containing the Lie algebras satisfying the condition: ``the pair of derivations of the extension 
contains at least one inner derivation".
For this subclass, we have:
\begin{thm}\label{Thm3}
	For an arbitrary $n$-dimensional nilpotent Lie algebra $H$, the problem of classifying all Lie algebras 
	in {\Lntwoa} with the derived algebra $H$ is equivalent to the problem of classifying 
	outer derivations in the first cohomology space $H^1(\R \oplus H,\R \oplus H)$ satisfying equivalent conditions in Proposition \ref{Pro4.1}, up to proportional similarity. 
\end{thm}

In practice, the first cohomology space $H^1(K,K)$ of an $n$-dimensional Lie algebra $K$ can be viewed as a subspace of $\M_n(\R)$, and it can be determined effectively whenever a basis for $K$ is fixed.
We can classify elements in $H^1(K,K)$ by simply using Jordan canonical forms (JCFs).
Therefore, Theorems \ref{Thm1} and \ref{Thm3} allow us to classify 
{\Lnone} and {\Lntwoa} whenever a classification of $n$-dimensional real nilpotent Lie algebras is given. 
Since the classification of real nilpotent Lie algebras was obtained up to dimension 7 (see Gong \cite{Gon98}),  
these theorems can be applied to achieve full classifications of $\mathrm{Lie}(n+1,n)$ and $\mathrm{Lie}_{\ad}(n+2,n)$ with $n \leq 7$.
It is noted that full classifications for {\Lnone} with $n \leq 5$ can be read from \cite{SW14}, while those with $n = 6,7$ are not reported so far.
In addition, we illustrate Theorems \ref{Thm1} and \ref{Thm3} by re-classifying {\Lnone} and {\Lntwoa} in cases $n=3,4$.
The obtained classifications are more compact than those in \cite{Mub63a, Mub63b} in the sense that unnecessary classes in \cite{Mub63a, Mub63b} are removed. 

The paper is organized as follows. 
Section~\ref{sec2} recalls necessary definitions and facts from Lie algebras.
In Section \ref{sec3}, we prove Theorems \ref{Thm1} and give two examples to illustrate the classification method derived from the theorem.
Afterwards, Theorem \ref{Thm2} is proved in Section \ref{sec4}. Finally, we prove Theorem \ref{Thm3} and classify {\Lntwoa} 
in low dimensions in Section \ref{sec5}.

\section{Preliminaries}\label{sec2}

Throughout this paper, the underlying field is always the field $\R$ of real numbers 
and $n \geq 2$ is a positive integer number. We use the following notations:
\begin{itemize}
	\item $\s \{x_1, \ldots, x_n\}$ is the vector space with basis $\{x_1, \ldots, x_n\}$.
	\item For two vector subspaces $U$ and $V$, we denote their sum and direct sum
	by $U+V$ and $U \dotplus V$, respectively.
	\item For a Lie algebra $L$, the notations $L^1 \co [L, L]$ and $L^2 \co \left[{L^1, L^1}\right]$ 
	denote the first and second derived ideals in its derived series, respectively. 
	Note that $L^1$ is also called the \emph{derived algebra} of $L$. 
	Besides, $\Z(L)$, $\D (L)$, $\ad (L)$ and $\Au(L)$ indicate the center, the set of derivations, 
	inner derivations and automorphisms of $L$, respectively.
	\item $a_x \co {\ad_x|}_{L^1}$ is the restriction of the adjoint operator $\ad_x$ on $L^1$. 
	\item $\M_n(\F)$ is the set of $n$-square matrices with entries in a field $\F$ 
	and $\GL_n(\F)$ denotes the group of all invertible matrices in $\M_n(\F)$.
	\item $H \oplus K$ is the direct sum of the Lie algebras $H$ and $K$, whereas $\R z \oplus_d K$ 
	is the semi-direct sum of Lie algebras by a certain $d \in \D (K)$, i.e. $[z, x] = d(x)$ for all $x \in K$.
	\item $\h_3 = \s \left\lbrace x_1, x_2, x_3 \, \big| \, [x_2, x_3] = x_1 \right\rbrace$ is the 
   \emph{3-dimensional real Heisenberg Lie algebra}. 
\end{itemize}

\begin{defn}\label{Def2.1}
Let $\F$ be a field.
	\begin{enumerate}
		\item[1.] 	Two matrices $A, B \in \M_n(\F)$ are \emph{proportionally similar}, denoted by $A \sim_p B$, 
		if there exist $c \in \F \setminus \{0\}$ and $C \in \GL_n(\F)$ such that $cA = C^{-1}BC$.
		
		\item[2.] 	Two endomorphisms $\varphi, \psi$ of a finite dimensional vector space $V$ over $\F$ are \emph{proportionally similar}, denoted by $\varphi \sim_p \psi$, if there exist $c \in \F \setminus \{0\}$ and an automorphism $\sigma$ of $V$ such that $c \varphi = {\sigma}^{-1}\psi \sigma$.  	
	\end{enumerate}
\end{defn}

\begin{rem}\label{Remark2.2}
		Two endomorphisms of a finite dimensional vector space over $\F$ are proportionally similar if and only if their matrices with respect to a given basis are proportionally similar.	
The classification of $\M_n(\F)$ up to proportional similarity can be derived from the JCFs of square matrices.	
\end{rem}

\begin{defn}[{\cite[Introduction]{FKPS18}}]
	A classification problem is called \emph{wild} if it contains the problem of classifying 
	pairs of $n \times n$ matrices up to similarity transformations 
	$(M,N) \mapsto S^{-1}(M,N)S \co \left(S^{-1}MS, S^{-1}NS\right)$ with nonsingular $S$.
\end{defn}

As mentioned in Section \ref{sec1}, this concept was introduced by Donovan and Freislich \cite{DF72,DF73}. 
Moreover, wild problems are considered as hopeless since it contains the problem of classifying 
an arbitrary system of linear mappings, that is, representations of an arbitrary quiver (see \cite{FKPS18}).
Now we recall the notion of \emph{weak similarity}.
	
\begin{defn}\label{Definition2.4}
Let $\F$ be a field.
		\begin{enumerate}
		\item\label{Def2.4-1} Let $A, B, A' , B' \in \M_n (\F)$. Two matrix pairs $(A, B)$ and $(A', B')$ are called \emph{weakly similar} if there exist $S \in \GL_n(\F)$	and $\begin{bmatrix} \alpha & \beta \\ \gamma & \delta \end{bmatrix} \in \GL_2 (\F)$ such that 
		$$(A', B') = S^{-1}(\alpha A + \beta B, \gamma A + \delta B)S.$$
		
		\item Two pairs of endomorphisms $({\varphi}_1, {\psi}_1)$, $({\varphi}_2, {\psi}_2)$ of a finite dimensional vector space $V$ over $\F$ are called \emph{weakly similar} if there exist an automorphism $\sigma$ of $V$ and $\begin{bmatrix} \alpha & \beta \\ \gamma & \delta \end{bmatrix} \in \GL_2 (\F)$ such that 
		$$({\varphi}_2, {\psi}_2) = {\sigma}^{-1}(\alpha {\varphi}_1 + \beta {\psi}_1, \gamma {\varphi}_1 + \delta {\psi}_1)\sigma .$$
	\end{enumerate}
\end{defn}

\begin{rem}\label{Remark2.5}
   The first item in Definition \ref{Definition2.4} was taken in \cite[Introduction]{FKPS18}. Two pairs of endomorphisms of a finite dimensional vector space over $\F$ are weakly similar if and only if two pairs of their matrices with respect to a given basis are weakly similar. In \cite[Theorem 1]{FKPS18}, the authors proved that the problem of classifying pairs of commuting matrices up to weak similarity is wild.
\end{rem}

In the rest of this section, we recall some well-known facts and present some preliminary results before entering to the main results.

\begin{prop}[{\cite[Lemma 2.1]{Gra05}}]\label{Lem2.1-deGraaf}
   Let $L$ be a real solvable Lie algebra. Then there is a subalgebra $K \subset L$ of codimension 1 
   and a derivation $d$ of $K$ such that $L = \R x \oplus_d K$. Moreover, if $L$ is not abelian, 
   then $d$ and $K$ can be chosen such that $d$ is an outer derivation of $K$. 
\end{prop}

According to Proposition \ref{Lem2.1-deGraaf}, each real solvable Lie algebra can be seen as 
an extension of a 1-codimensional subalgebra by a suitable derivation.

\begin{prop}[{\cite[Lemma 2.2]{Gra05}}]\label{Lem2.2-deGraaf}
   Let $K$ be a solvable Lie algebra and $d_1, d_2$ derivations of $K$. Set $L_i = \R x \oplus_{d_i} K$ for $i=1,2$. 
   Suppose that there exists $\sigma \in \Au \left(K\right)$ such that $\sigma d_1 \sigma^{-1} = \lambda d_2$ 
   for some scalar $\lambda \neq 0$. Then $L_1$ and $L_2$ are isomorphic.
\end{prop}

The result in Proposition \ref{Lem2.2-deGraaf} is a sufficient condition, not a necessary one. 
Now, by Propositions \ref{Lem2.1-deGraaf} and \ref{Lem2.2-deGraaf}, we get some initial results as follows.

\begin{prop}\label{Pro2.8}
   Let $K$ be a real solvable Lie algebra and $d_1, d_2 \in \D(K)$. Set $L_i \co \R x \oplus_{d_i} K$ for $i=1,2$. 
   If $d_1 = d_2 + \ad_u$ for some $u \in K$, then $L_1$ and $L_2$ are isomorphic.
\end{prop}
	   
\begin{proof}
	We define a map $\ts \colon \R x \oplus_{d_1} K \to \R x \oplus_{d_2} K$ such that:
	\[
		\begin{cases}
			\ts(y) = y, & y \in K, \\
			\ts(x) = x + u.
		\end{cases}
	\]
	It is obvious that $\ts$ is a linear isomorphism. We show that it also preserves Lie brackets. 
	In fact, for arbitrary $y \in K$, we have
	\[
		\begin{array}{l l}
			& \ts \left([x, y]\right) = [\ts(x), \ts(y)] \\
			\Leftrightarrow & [x, y] = [\ts(x), \ts(y)] \\
			\Leftrightarrow & d_1 (y) = [x + u, y] \\
			\Leftrightarrow & d_1(y) = (d_2 + \ad_u)(y).
		\end{array}
	\]
	The last equation is obvious which completes the proof.
\end{proof}

By combining Propositions \ref{Lem2.2-deGraaf} and \ref{Pro2.8}, we have the following corollary.

\begin{cor}\label{Cor2.9}
	Let $K$ be a real solvable Lie algebra and $d_1$, $d_2 \in \D(K)$. Set $L_i \co \R x \oplus_{d_i} K$ for $i=1,2$. 
	If there is $\sigma \in \Au(K)$ such that $\sigma d_1 \sigma^{-1} = \alpha d_2 + \ad_u$ for some 
	$\alpha \in \R \setminus \{0\}$ and $u \in K$, then $L_1$ and $L_2$ are isomorphic. 
\end{cor}

\begin{cor}\label{Cor2.10}
	Let $K$ be a real solvable Lie algebra and $d \in \D(K)$. Set $L \co \R x \oplus_d K$. 
	If $d \in \ad(K)$ then $L$ is decomposable. In addition, we can decompose $L$ 
	into the form $L = \R x' \oplus K$ for some $x' \in L \setminus K$.
\end{cor}

\begin{proof}
	If $d = \ad_u$ for some $u \in K$ then we change $x' = x - u$ to get $L = \R x' \dotplus K$. 
	Besides, we have $[x', z] = [x - u, z] = d(z) - \ad_u (z) = 0$ for all $z \in K$.
	Therefore, $L = \R x' \oplus K$ with $x' \in L \setminus K$.
\end{proof}

\section{The problem of classifying \Lnone}\label{sec3}

We present a proof for Theorem~\ref{Thm1} in this section.
Recall that {\Lnone} is the class of all $(n+1)$-dimensional real solvable Lie algebras 
having $n$-dimensional derived algebras.

\subsection{Description of \Lnone}\label{subsec3.1}

Since the derived algebra of a solvable Lie algebra is nilpotent (see \cite{Jac62}),
we can classify {\Lnone} by applying the following steps:
\begin{enumerate}[Step 1.]
	\item Take an $n$-dimensional real nilpotent Lie algebra $K$;
	\item\label{step2-Lie(n+1,n)} Extend $K$ to all $(n+1)$-dimensional real solvable 
	Lie algebras $L$ which admits $K$ as derived algebra, i.e. $L^1 = K$. 
	Afterward, classify such Lie algebras;
	\item\label{step3-Lie(n+1,n)} Repeat two steps above for all possibilities of $K$.
\end{enumerate}

First of all, we fix an arbitrary $n$-dimensional real nilpotent Lie algebra $K$. Assume that $L \in \Li$ with $L^1 = K$. 
Without loss of generality, we can choose a basis $\{x_1, \ldots, x_n, y\}$ of $L$ in which $L^1 = K = \s \{x_1, \ldots, x_n \}$. 
By Proposition \ref{Lem2.1-deGraaf}, there exists $d \in \D(K)$ satisfying
\[
	L = \R y \oplus_d K.
\]
It also notes that $d(u) = [y, u]$ for $u \in K$, i.e. $d = a_y$.

In general, a Lie algebra of the form $L = \R y \oplus_d K$ for some $d \in \D(K)$ does not necessarily belong to \Lnone. 
Therefore, we first point out a necessary and sufficient condition of $d$ such that $L$ belongs to \Lnone.

\begin{prop}\label{Pro3.1}
	Let $K$ be an $n$-dimensional real nilpotent Lie algebra, and $L = \R y \oplus_d K$ for $d \in \D (K)$ as above. 
	By renumbering, if necessary, we can assume that $K^1 = \s \{x_1, \ldots, x_m\}$ for $0 \leq m < n$. 
	Then the following conditions are equivalent:
	\begin{enumerate}
		\item $L \in \Li$;
		\item $L^1 = K$;
		\item $\s \{x_{m+1}, \ldots, x_n\} \subset d(K) + K^1$;
		\item\label{Cond3-Pro3.1} $\ra \left(d_{ij}\right)_{i > m} = n - m$, where $\left(d_{ij}\right)_n \in \M_n (\R)$ 
		is the matrix of $d$ with respect to the basis $\{x_1, \ldots, x_n\}$;
		\item $\td \colon K/K^1 \to K/K^1$ is a linear isomorphism, where $\td$ is induced from $d$ on the quotient Lie algebra $K/K^1$.
	\end{enumerate}
\end{prop}
   
\begin{proof}
	It is obvious that
	\[
		L^1 = \s \{[y, u] \, \big| \, u \in K\} + \s \{[v, u] \, \big| \, v, u \in K\} = d(K) + K^1 \subset K.
	\]
	Therefore, $L \in \Li$ if and only if
	\[
		\begin{array}{l l}
			& L^1 = K \\
			\Leftrightarrow & d(K) + K^1 = K \\
			\Leftrightarrow & d(K) + K^1 = \s \{x_{m+1}, \ldots, x_n\} + K^1 \\
			\Leftrightarrow & \s \{x_{m+1}, \ldots, x_n\} \subset d(K) + K^1 \quad \left(\text{since $\s \{x_1, \ldots, x_m\} = K^1$}\right) \\
			\Leftrightarrow & \ra \left(d_{ij}\right)_{i > m} = n - m.
		\end{array}
	\] 
	Due to the Leibniz formula, the derived algebra $K^1$ is $d$-invariant, i.e. $d\left(K^1\right) \subset K^1$. 
	Hence, $\td \colon K/K^1 \to K/K^1$ is well-defined and the last equation above means that $\td$ is a linear isomorphism. 
	This completes the proof.
\end{proof}

\begin{prop}\label{Pro3.2}
	Let $K$ be an $n$-dimensional real nilpotent Lie algebra. Then we have the following assertions:
	\begin{enumerate}
		\item\label{part1-Pro3.2} All Lie algebras in {\Lnone} are indecomposable.
		\item No $L \in \Li$ is an extension of $K$ by an inner derivation $d$ of $K$. 
		In other words, if $d \in \ad(K)$ then $L = \R y \oplus_d K \notin \Li$.
	\end{enumerate}
\end{prop}

\begin{proof}
	\begin{enumerate}
		\item Let $L \in \Li$. If $L$ is decomposable then
		\[
			L = L_1 \oplus L_2
		\]
		which implies $L^1 = L_1^1 \oplus L_2^1$ and $\dim L^1 = \dim L_1^1 + \dim L_2^1$. 
		Because $L_1$ and $L_2$ are solvable, we have
		\[
			\begin{cases}
				\dim L_1^1 \leq \dim L_1 - 1 \\ \dim L_2^1 \leq \dim L_2 - 1
			\end{cases}
		\]
		and thus $\dim L^1 \leq \dim L_1 + \dim L_2 - 2 = \dim L - 2$. This contradicts the assumption $L \in \Li$. 
		Therefore, $L$ is indecomposable.
		
		\item A direct consequence of Corollary \ref{Cor2.10} and part \ref{part1-Pro3.2} above.
	\end{enumerate}
	\noindent The proof of Proposition \ref{Pro3.2} is complete.
\end{proof}

\subsection{Proof of Theorem \ref{Thm1}}

As mentioned in Subsection \ref{subsec3.1}, the classification of {\Lnone} consists of 
three steps initializing by an $n$-dimensional real nilpotent Lie algebra $K$. 
Results in this subsection perform Step \ref{step2-Lie(n+1,n)}, that is, to extend $K$ to all 
$L \in \Li$ with $L^1 = K$ and then classify such Lie algebras.
 
According to Proposition \ref{Pro3.1}, the problem of classifying $L \in \Li$ with $L^1 = K$ 
is reduced to find out the conditions of $d_1$ and $d_2$ satisfying Proposition \ref{Pro3.1} 
such that two Lie algebras $\R y \oplus_{d_1} K$ and $\R y \oplus_{d_2} K$ are isomorphic. 
Namely, we have the following result.
	
\begin{prop}\label{Pro3.3}
	Let $L_1 = \R y \oplus_{d_1} K$ and $L_2 = \R y \oplus_{d_2} K$ be extensions of $K$ 
	by outer derivations $d_1$ and $d_2$ which satisfy Proposition \ref{Pro3.1}, respectively. 
	Then $L_1$ and $L_2$ are isomorphic if and only if there exist $\alpha \neq 0$ and $\sigma \in \Au(K)$ 
	such that $\sigma d_1 \sigma^{-1} = \alpha d_2 + \ad_u$ for some $u \in K$.
\end{prop}

\begin{proof}
	\begin{enumerate}
		\item[$(\Leftarrow)$] It follows directly from Corollary \ref{Cor2.9}.
		
		\item[$(\Rightarrow)$] Assume that $L_1$ and $L_2$ are isomorphic by $\ts \colon \R y \oplus_{d_1} K \to \R y \oplus_{d_2} K$. 
		By Proposition \ref{Pro3.1}, $L_1^1=L_2^1=K$.
		Thus $\sigma \co \ts|_K \colon K \to K$ is also an isomorphism. 
		Set $\ts(y) = \alpha y + u$ with $\alpha \neq 0$ and $u \in K$. 
		Then
	\[
		\begin{array}{l l l}
			& \ts ([y, x]) = \left[ \ts(y), \ts(x)\right], &\forall x \in K\\
			\Leftrightarrow & \sigma \left([y, x]\right) = \left[ \alpha y + u, \sigma (x)\right], &\forall x \in K\\
			\Leftrightarrow & \sigma d_1(x) = \left(\alpha d_2 + \ad_u\right) \sigma (x), &\forall x \in K\\
			\Leftrightarrow & \sigma d_1 \sigma^{-1} = \alpha d_2 + \ad_u.
		\end{array}
	\]
	\end{enumerate}
	\noindent The proof of Proposition \ref{Pro3.3} is complete.
\end{proof}

\begin{pot1}
	As we have seen, for a given $n$-dimensional real nilpotent Lie algebra $K$, 
	all $L \in \Li$ with $L^1 = K$ are of the forms $\R y \oplus_d K$ in which 
	$d \in \D(K) \setminus \ad(K)$ satisfies one of the equivalent conditions in Proposition \ref{Pro3.1}. 
	Now, let us consider two extensions
	\[
		\begin{array}{l l l}
			L_i \co \R y \oplus_{d_i} K, & i = 1,2,
		\end{array}
	\]
	where $d_1, d_2 \in \D(K) \setminus \ad(K)$ satisfy equivalent conditions in Proposition \ref{Pro3.1}.
	By Proposition \ref{Pro3.3}, $L_1 \cong L_2$ if and only if there exist $\alpha \neq 0$ and 
	$\sigma \in \Au (K)$ such that $\sigma d_1 \sigma^{-1} = \alpha d_2 + \ad_u$ for some $u \in K$. 
	By setting $d = \frac{1}{\alpha} \sigma d_1 \sigma^{-1}$ we have:
	\begin{itemize}
		\item $d \sim_p d_1$;
		\item $d = d_2 + \ad_{\frac{1}{\alpha}u}$, i.e. $\bd = \bd_2$ in $\D(K)/ \ad(K)$, where $\bd$ 
		and $\bd_2$ are $\ad(K)$-modulo classes of $d, d_2 \in \D(K) \setminus \ad(K)$, respectively.
	\end{itemize}
	This implies $\bd_1 \sim_p \bd_2$ in $\D (K)/ \ad(K)$. For convenience, we also call the 
	$\ad (K)$-modulo class $\bd$ of $d \in \D(K) \setminus \ad(K)$ an \emph{outer derivation} of $K$.

	It is noted that $K$ can be viewed as a left $K$-module with the scalar product defined by the Lie bracket $x \cdot y:=[x,y]$ for each $(x,y) \in K \times K$.
	From this point of view, by using \cite[VII, Proposition~2.2]{Hilton1997}, we can identify $\D(K)/ \ad(K)$ with $\hk$ which is the first cohomology space of $K$ with coefficients in $K$.
	Therefore, the problem of classifying all $L \in \Li$ with $L^1 = K$ is equivalent to the problem of classifying all 
	equivalent classes $\bd \in \hk$ of outer derivations $d$ satisfying equivalent conditions in 
	Proposition \ref{Pro3.1} up to proportional similarity. The proof is complete.\hfill{$\square$}
\end{pot1}

\begin{rem}\label{Rem3.4}
	\begin{enumerate}
		\item When $L^1 = K = \R^n$, the structure of $L = \R y \oplus_d \R^n$ is absolutely determined 
		by the derivation $d \in \D(\R^n)$. According to Propositions \ref{Pro3.1} and \ref{Pro3.3}, 
		the classification problem is now equivalent to the classification of $\GL_n (\R)$ up to proportional similarity
		which is reduced to classify JCFs in $\GL_n (\R)$ by Remark \ref{Remark2.2}. 
		
		\item More generally, the results in this section are also true if we start with ``$K$ is solvable'' instead of ``$K$ is nilpotent''.
	\end{enumerate}
\end{rem}

\begin{rem}\label{Rem3.5}
	We also note that the problem of classifying {\Lnone} is essentially wild because Step \ref{step3-Lie(n+1,n)} 
	in Subsection \ref{subsec3.1} involves the problem of classifying $n$-dimensional real nilpotent Lie algebras. 
	However, we can achieve a full classification of {\Lnone} whenever we have a classification of $n$-dimensional real nilpotent Lie algebras.
\end{rem}

\subsection{Illustrative examples}

In this subsection, we will classify {\Lnone} in low dimensions by the technique pointed out above. 
More concretely, we give classification of {\Lnone} for $n \leq 4$. Because the case $n = 2$ is a 
simple result of classifying 3-dimensional Lie algebras (see, e.g. \cite[Theorem 10]{VTTTT19}), 
we will illustrate with $n = 3$ by Example \ref{Ex3.6} and $n = 4$ by Example \ref{Ex3.7}.

\begin{ex}[The classification of $\mathrm{Lie}(4,3)$]\label{Ex3.6}
	It is well-known that there are only two 3-dimensional real nilpotent Lie algebras: $\R^3$ and $\h_3$. 
  
	\begin{enumerate}[\bf A.]
		\item If $K = \R^3$, it returns to classify JCFs in $\GL_3 (\R)$ by Remark \ref{Rem3.4}. 
		More concretely, there are four families of Lie algebras in $\mathrm{Lie}(4,3)$ 
		with derived algebra $\R^3$ listed in \cite[Proposition 3.2]{VHTHT16}.
      
		\item For $K = \h_3 = \s \{x_1, x_2, x_3 \, \big| \, [x_2, x_3] = x_1\}$, we set
   		\[
   			\begin{array}{l l}
   				L_d \co \R x_4 \oplus_d K, & d \in \D(K).
   			\end{array}
   		\]
		Since $d([x_i, x_j]) = [d(x_i), x_j] + [x_i, d(x_j)]$ for $1 \leq i < j \leq 3$, 
		the matrix of $d$ with respect to the basis $\{x_1, x_2, x_3\}$ is of the form
		\[
			\begin{array}{l l}
				d = \begin{bmatrix} a+b & f & g \\ 0 & a & c \\ 0 & e & b \end{bmatrix}; & a, b, c, e, f, g \in \R.
			\end{array}
		\]
		Since
		\[
			\ad(K) = \s \left\lbrace \ad_{x_2} = \begin{bmatrix} 0 & 0 & 1 \\ 0 & 0 & 0 \\ 0 & 0 & 0 \end{bmatrix},
			\ad_{x_3} = \begin{bmatrix} 0 & -1 & 0 \\ 0 & 0 & 0 \\ 0 & 0 & 0 \end{bmatrix} \right\rbrace
		\]
		we have
		\[
			\hk = \left\lbrace \bd = \begin{bmatrix} a+b & 0 & 0 \\ 0 & a & c \\ 0 & e & b \end{bmatrix} 
			\Bigg| \, a, b, c, e \in \R \right\rbrace.
		\]
		Since $K^1 = \s \{x_1\}$, condition \ref{Cond3-Pro3.1} of Proposition \ref{Pro3.1} implies that
		\begin{equation}\label{Cond3.1}
			\begin{array}{l l l l l}
				L \in \mathrm{Lie}(4,3) & \Leftrightarrow & \ra \begin{bmatrix} a & c \\ e & b \end{bmatrix} = 2
				& \Leftrightarrow & \det \begin{bmatrix} a & c \\ e & b \end{bmatrix} \neq 0.
			\end{array}
		\end{equation}
		Therefore, the classification of $L_d$ is equivalent to classify $\bd \in \hk$ with condition \eqref{Cond3.1} 
		up to proportional similarity. The possible JCFs of $\bd$ are as follows (we omit zeros):
		\[
			\begin{array}{l}
				J_{1(\lambda_1, \lambda_2)} = \di (\lambda_1+\lambda_2, \lambda_1, \lambda_2) \; (|\lambda_1| \geq |\lambda_2| > 0), \\
				J_{2(\lambda)} = \begin{bmatrix} 2\lambda \\ & \lambda & 1 \\ & & \lambda \end{bmatrix} (\lambda \neq 0), \\
				J_{3(\lambda)} = \begin{bmatrix} 2\lambda \\ & \lambda & 1 \\ & -1 & \lambda \end{bmatrix} (\lambda \in \R).
			\end{array}
		\]
		It can check that
		\[
			\begin{array}{l}
				J_{1(\lambda_1, \lambda_2)} \sim_p \di (1+\lambda, 1, \lambda) \co A (\lambda) \, (0 < |\lambda| \leq 1), \\
				J_{2(\lambda)} \sim_p \begin{bmatrix} 2 \\ & 1 & 1 \\ & & 1 \end{bmatrix} \co B, \\
				J_{3(\lambda)} \sim_p \begin{bmatrix} 2|\lambda| \\ & |\lambda| & 1 \\ & -1 & |\lambda| \end{bmatrix} \co C(\lambda) \, (\lambda \geq 0),	
			\end{array}
		\]
		i.e. we have three equivalent classes of $\bd \in \hk$ up to proportional similarity. 
		Therefore, there exist three families of Lie algebras 
		in $\mathrm{Lie}(4,3)$ with derived algebra $K = \h_3$, namely, $L_{A(\lambda)}$, $L_B$ and $L_{C(\lambda)}$.
	\end{enumerate}
\end{ex}

\begin{ex}[The classification of $\mathrm{Lie}(5,4)$]\label{Ex3.7}
	If $L \in \mathrm{Lie}(5,4)$ then $L^1$ falls into three cases: 
	$\R^4$, $\R \oplus \h_3$ and $\g_4$ (see \cite[Proposition 1]{Dix58}). 		
	\begin{enumerate}[\bf A.]
		\item If $K = \R^4$ it is the classification of JCFs in $\GL_4 (\R)$. More concretely, 
		there are fourteen families of Lie algebras in $\mathrm{Lie}(5,4)$ with derived algebra 
		$\R^4$ listed in \cite[Proposition 3.4]{VHTHT16}.
    	
		\item Let $K = \R \oplus \h_3 = \s \left\lbrace x_1, x_2, x_3, x_4 \, \big| \, [x_2, x_3] = x_1\right\rbrace$. Set
		\[
			\begin{array}{l l}
				L = L_d \co \R x_5 \oplus_d K, & d \in \D(K).
			\end{array}
		\]
		By the same argument in Example~\ref{Ex3.6}, we have
		\[
			\hk = \left\lbrace \bd = \begin{bmatrix} a+b & 0 & 0 & k \\ 0 & a & e & 0 \\ 0 & f & b & 0 \\ 0 & g & h & c \end{bmatrix}
			\Bigg| \, a, b, c, e, f, g, h, k \in \R \right\rbrace.
		\]
		Since $K^1 = \s \{x_1\}$, condition \ref{Cond3-Pro3.1} of Proposition \ref{Pro3.1} implies that
		\begin{equation}\label{Cond3.2}
			\begin{array}{l l l l l}
				L \in \mathrm{Lie}(5,4) & \Leftrightarrow & \ra \begin{bmatrix} a & e & 0 \\ f & b & 0 \\ g & h & c \end{bmatrix} = 3
				& \Leftrightarrow & \begin{cases} c \neq 0 \\ \det \begin{bmatrix} a & e \\ f & b \end{bmatrix} \neq 0 \end{cases}\hspace{-10pt}
			\end{array}
		\end{equation}
		Therefore, the classification of $L_d$ is equivalent to classify $\bd \in \hk$ 
		with condition \eqref{Cond3.2} up to proportional similarity.
		
		Set $D \co \begin{bmatrix} a & e & 0 \\ f & b & 0 \\ g & h & c \end{bmatrix}$. 
		The possible JCFs of $D$ are as follows:
		\[
			\begin{array}{l l l}
				J_1 = \di (\lambda_1, \lambda_2, c) \, (\lambda_1\lambda_2c \neq 0) 
					\sim_p \di (1, \alpha, \beta) \, (\alpha\beta \neq 0), \\
				J_2 = \begin{bmatrix} \lambda \\ & c \\ & 1 & c \end{bmatrix} (\lambda c \neq 0) 
					\sim_p \begin{bmatrix} \alpha \\ & 1 \\  & 1 & 1 \end{bmatrix} (\alpha \neq 0), \\
				J_3 = \begin{bmatrix} \lambda \\ 1 & \lambda \\ & & c \end{bmatrix} (\lambda c \neq 0) 
					\sim_p \begin{bmatrix} 1 \\ 1 & 1 \\ & & \beta \end{bmatrix} (\beta \neq 0), \\
				J_4 = \begin{bmatrix} c \\ 1 & c \\ & 1 & c \end{bmatrix} (c \neq 0)
					\sim_p \begin{bmatrix} 1 & &  \\ 1 & 1 & \\  & 1 & 1 \end{bmatrix}, \\
				J_5 = \begin{bmatrix} \lambda & 1 \\ -1 & \lambda \\ & & c \end{bmatrix} (\lambda \in \R, c \neq 0).
			\end{array}
		\]
		We would like to emphasize two points here:
		\begin{itemize}
			\item The equivalent classes of $\bd$ derived from $J_2$ and $J_3$ are absolutely 
			different even though $J_2$ and $J_3$ coincide from Linear Algebra point of view. 
			The reason is that the pair $(x_2, x_3)$ enters into $K$ symmetrically, 
			whereas $x_4$ is apart from other basic vectors of $L^1$. In other words, 
			the location of Jordan blocks in the JCFs of $D$ is significant.
			
			\item In this case, we should not represent the JCFs of $D$ with entries 1 
			above the main diagonal. For instance, we cannot represent
			\[
				\begin{array}{l l l}
					J_2 = \begin{bmatrix} \lambda \\ & c & 1 \\ & & c \end{bmatrix} (\lambda c \neq 0)
					& \text{or} & J_4 = \begin{bmatrix} c & 1 \\ & c & 1 \\ & & c \end{bmatrix}
				\end{array}
			\]
			because $\bd \in \hk$ has two zeros in the last columns. In fact, if we represent $J_2$ and $J_4$ 
			as these types, the obtained results are not Lie algebras since they do not obey the Jacobi identity.
		\end{itemize}
		All above arguments show that we have five equivalent classes of $\bd \in \hk$ up to proportional similarity as follows:
		\[
			\begin{array}{l l}
				D_1 = \begin{bmatrix} 1+\alpha & & & k \\ & 1 \\ & & \alpha \\ & & & \beta \end{bmatrix} (\alpha\beta \neq 0), &
				D_2 = \begin{bmatrix} 1+\alpha & & & k \\ & \alpha \\ & & 1 \\ & & 1 & 1 \end{bmatrix} (\alpha \neq 0), \\
				D_3 = \begin{bmatrix} 2 & & & k \\ & 1 \\ & 1 & 1 \\ & & & \beta \end{bmatrix} (\beta \neq 0), &
				D_4 = \begin{bmatrix} 2 & & & k \\ & 1 \\ & 1 & 1 \\ & & 1 & 1 \end{bmatrix}, \\
				D_5 = \begin{bmatrix} 2 \lambda & & & k \\ & \lambda & 1 \\ & -1 & \lambda \\ & & & c \end{bmatrix} (\lambda \in \R, c \neq 0).
			\end{array}
		\]
		Finally, we need to refine parameters in each case.
		\begin{enumerate}
			\item For $\bd = D_1$. If $k = 0$ then we have $L_{A(\alpha, \beta)}$ where
			\[
				A(\alpha, \beta) \co \di \left(1 + \alpha, 1, \alpha, \beta\right), \quad \alpha\beta \neq 0.
			\]
			Since the transformation $T = \di \left(1, \begin{bmatrix} 0 & -1 \\ 1 & 0 \end{bmatrix}, \frac{1}{\alpha}, \alpha \right)$ 
			gives rise to the isomorphism $L_{A(\alpha, \beta)} \cong L_{A\left(\frac{1}{\alpha}, \frac{\beta}{\alpha}\right)}$, 
			we can reduce the parameters to $0 \neq |\alpha| \leq 1$ and $\beta \neq 0$. If $k \neq 0$, we eliminate $k$ by changing 
			$x'_4 = \frac{k}{\beta-1-\alpha}x_1 + x_4$ when $\beta \neq 1 + \alpha$, and normalize $k = 1$ 
			by scaling $x'_4 = \frac{1}{k}x_4$ when $\beta = 1+\alpha \neq 0$. Thus, we have $L_{B(\alpha)}$ where
			\[
				B(\alpha) \co \begin{bmatrix} 1 + \alpha & & & 1 \\ & 1 \\ & & \alpha \\ & & & 1+\alpha \end{bmatrix}, \quad \alpha \neq 0, -1.
			\]
			Similarly, we can also reduce the parameter to $\alpha \in (-1, 1] \setminus \{0\}$ as 
			$L_{B(\alpha)} \cong L_{B\left(\frac{1}{\alpha}\right)}$ by the isomorphism 
			$T = \di \left(1, \begin{bmatrix} 0 & -1 \\ 1 & 0 \end{bmatrix}, \frac{1}{\alpha}, \alpha \right)$.
         
			\item For $\bd = D_2$, we eliminate $k$ by changing $x'_4 = -\frac{k}{\alpha}x_1 + x_4$ to get $L_{C(\alpha)}$ where
			\[
				C(\alpha) \co \begin{bmatrix} 1+\alpha \\ & \alpha \\ & & 1 \\ & & 1 & 1 \end{bmatrix}, \quad \alpha \neq 0.
			\]
			
			\item For $\bd = D_3$. If $k = 0$ we have $L_{D(\beta)}$ where
			\[
				D(\beta) \co \begin{bmatrix} 2 \\ & 1 \\ & 1& 1 \\ & & & \beta \end{bmatrix}, \quad \beta \neq 0.
			\]
			If $k \neq 0$, we eliminate $k$ by changing $x'_4 = \frac{k}{\beta - 2}x_1 + x_4$ when $\beta \neq 2$, 
			and normalize $k = 1$ by scaling $x'_4 = \frac{1}{k}x_4$ when $\beta = 2$. Thus we have $L_E$ with
			\[
				E \co \begin{bmatrix} 2 & & & 1 \\ & 1 \\ & 1 & 1 \\ & & & 2 \end{bmatrix}.
			\]
			
			\item For $\bd = D_4$, we eliminate $k$ by changing $x'_4 = -kx_1 + x_4$ to get $L_F$ where
			\[
				F \co \begin{bmatrix} 2 \\ & 1 \\ & 1 & 1 \\ & & 1 & 1 \end{bmatrix}.
			\]
			
			\item For $\bd = D_5$. If $k = 0$ we have $L_{G(\lambda, c)}$ where
			\[
				G(\lambda, c) \co \begin{bmatrix} 2\lambda \\ & \lambda & 1 \\ & -1 & \lambda \\ & & & c \end{bmatrix}, \quad \lambda \in \R, c \neq 0.
			\]
			Since the transformation $T = \di (-1, -1, 1, 1, -1)$ creates the isomorphism $L_{G(\lambda, c)} \cong L_{G(-\lambda, -c)}$,
			we can reduce the parameters to $\lambda \geq 0$ and $c > 0$. If $k \neq 0$, we eliminate $k$ 
			by changing $x'_4 = \frac{k}{c - 2\lambda}x_1 + x_4$ when $c \neq 2\lambda$, and normalize $k = 1$ 
			by scaling $x'_4 = \frac{1}{k}x_4$ when $c = 2\lambda \neq 0$. Thus, we have $L_{H(\lambda)}$ where
			\[
				H(\lambda) \co \begin{bmatrix} 2\lambda & & & 1 \\ & \lambda & 1 \\ & -1 & \lambda \\ & & & 2\lambda \end{bmatrix}, 
				\quad \lambda \neq 0.
			\]
			Similarly, we can reduce the parameter to $\lambda > 0$ since $L_{H(\lambda)} \cong L_{H(-\lambda)}$ 
			by the isomorphism $T = \di (-1, -1, 1, 1, -1)$.
		\end{enumerate}
		To sum up, we have eight families of Lie algebras in $\mathrm{Lie}(5,4)$ with derived algebra 
		$K = \R \oplus \h_3$, namely, $L_{A(\alpha, \beta)}$, $L_{B(\alpha)}$, $L_{C(\alpha)}$, $L_{D(\beta)}$, 
		$L_E$, $L_F$, $L_{G(\lambda, c)}$ and $L_{H(\lambda)}$.
		
		\item Let $K = \g_4 = \s \{x_1, x_2, x_3, x_4 \, \big| \, [x_2, x_4] = x_1, [x_3, x_4] = x_2\}$. Set
		\[
			\begin{array}{l l}
				L = L_d \co \R x_5 \oplus_d K, & d \in \D(K).
			\end{array}
		\]
		In this case, we have
		\[
			\hk = \left\lbrace \bd = \begin{bmatrix} a+2b & 0 & e & 0 \\ 0 & a + b & 0 & 0 \\ 0 & 0 & a & c \\ 0 & 0 & 0 & b \end{bmatrix} 
			\Bigg| \, a, b, c,e \in \R \right\rbrace.
		\]
		Since $K^1 = \s\{x_1, x_2\}$, condition \ref{Cond3-Pro3.1} of Proposition \ref{Pro3.1} implies that 
		$L \in \mathrm{Lie}(5,4)$ if and only if $ab \neq 0$. So the classification of $L_d$ is equivalent to classify 
		$\bd \in \hk$ with $ab \neq 0$ up to proportional similarity. 
		The possible JCFs of $\begin{bmatrix} a & c \\ 0 & b \end{bmatrix}$ are as follows:
		\[
			\begin{array}{l l l}
				\begin{bmatrix} a \\ & b \end{bmatrix} (ab \neq 0) \sim_p \begin{bmatrix} \lambda \\ & 1 \end{bmatrix} (\lambda \neq 0)
				& \text{or} &
				\begin{bmatrix} a & 1 \\  & a \end{bmatrix} (a \neq 0) \sim_p \begin{bmatrix} 1 & 1 \\ & 1 \end{bmatrix}.
			\end{array}
		\]
		Similarly, we should not represent the JCFs in this case with entries 1 below the main diagonal. 
		Therefore, we have two equivalent classes of $\bd \in \hk$ up to proportional similarity as follows:
		\[
			\begin{array}{l l l}
				D_{1(\lambda)} = \begin{bmatrix} \lambda+2 & & e \\ & \lambda+1 \\ & & \lambda \\ & & & 1 \end{bmatrix} (\lambda \neq 0)
				& \text{or} & 
				D_2 = \begin{bmatrix} 3 & & e \\ & 2 \\ & & 1 & 1 \\ & & & 1 \end{bmatrix}.
			\end{array}
		\]
		We also eliminate $e$ in $D_{1(\lambda)}$ (resp. $D_2$) by changing $x'_3 = -\frac{e}{2}x_1 + x_3$ 
		(resp. $x'_3 = -\frac{e}{2}x_1 + x_3$ and $x'_4 = -\frac{1}{2}x_3 + x_4$) to obtain
		\[
			\begin{array}{l l l}
				I(\lambda) \co \di \left(\lambda+2, \lambda+1, \lambda, 1\right) \, (\lambda \neq 0)
				& \text{or} &
				J \co \begin{bmatrix} 3 \\ & 2 \\ & & 1 & 1 \\ & & & 1 \end{bmatrix}.
			\end{array}
		\]
		To summarize, there are two families of Lie algebras in $\mathrm{Lie}(5,4)$ with derived algebras 
		$K = \g_4$, namely, $L_{I(\lambda)}$ and $L_J$.
	\end{enumerate}
\end{ex}

\begin{rem}
	In 1963, Mubarakzyanov \cite{Mub63a,Mub63b} classified real solvable Lie algebras of dimension 4 and 5. 
	We summarize the correspondence between our classification and Mubarakzyanov's ones in Table \ref{tab1}.
	\begin{table}[!h]
		\centering
		\begin{tabular}{c c c l c c c c}
         \hline\noalign{\smallskip}
         \rotatebox[origin=c]{90}{Examples} & \rotatebox[origin=c]{90}{$\dim L$} & \rotatebox[origin=c]{90}{$[L, L]$} 
         & \rotatebox[origin=c]{90}{Types} & \rotatebox[origin=c]{90}{Notes} & \rotatebox[origin=c]{90}{\cite[\S 5]{Mub63a}} 
         & \rotatebox[origin=c]{90}{\cite[\S 10]{Mub63b}} & \rotatebox[origin=c]{90}{Notes} \\
         \noalign{\smallskip}\hline\noalign{\smallskip}
         \multirow{3}{*}{\ref{Ex3.6}} & \multirow{3}{*}{4} & \multirow{3}{*}{$\h_3$} & $L_{A(\lambda)}$ & $0 < |\lambda| \leq 1$ 
         & $g_{4,8}^h$ & & $0 \neq |h| \leq 1$ \\
         & & & $L_B$ & & $g_{4,7}$ & & \\
          & & & $L_{C(\lambda)}$ & $\lambda \geq 0$ & $g_{4,9}^p$ & & $p \neq 0$ \\
         \noalign{\smallskip}\hline\noalign{\smallskip}
         \multirow{10}{*}{\ref{Ex3.7}} &\multirow{10}{*}{5} & \multirow{8}{*}{$\R \oplus \h_3$} & $L_{A(\alpha, \beta)}$ 
         & $0 < |\alpha| \leq 1$, $\beta \neq 0$ & & $g_{5,19}^{\alpha\beta}$ & $\alpha\beta \neq 0$ \\
            & & & $L_{B(\alpha)}$ & $\alpha \in \left(-1, 1\right] \setminus \{0\}$ & & $g_{5,20}^\alpha$ & $\alpha \neq 0, -1$ \\
            & & & $L_{C(\alpha)}$ & $\alpha \neq 0$ & & $g_{5,28}^\alpha$ & $\alpha \neq 0$ \\
            & & & $L_{D(\beta)}$ & $\beta \neq 0$ & & $g_{5,23}^\beta$ & $\beta \neq 0$ \\
            & & & $L_E$ & & & $g_{5,24}^\epsilon$ & $\epsilon = \pm 1$ \\
            & & & $L_F$ & & & $g_{5,21}$ &  \\
            & & & $L_{G(\lambda, c)}$ & $\lambda \geq 0, c > 0$ & & $g_{5,25}^{p\beta}$ & $\beta \neq 0$ \\
            & & & $L_{H(\lambda)}$ & $\lambda > 0$ & & $g_{5,26}^{\epsilon p}$ & $\epsilon = \pm 1$ \\
         \noalign{\smallskip}\cline{3-8}\noalign{\smallskip}
            & & \multirow{2}{*}{$\g_4$} & $L_{I(\lambda)}$ & $\lambda \neq 0$ & & $g_{5,30}^h$ & $h \neq 0$ \\
            & & & $L_J$ & & & $g_{5,31}$ & \\
         \noalign{\smallskip}\hline
      \end{tabular}
      \caption{{\Lnone} and Mubarakzyanov \cite{Mub63a,Mub63b} in low dimensions.}\label{tab1} 
   \end{table}
   According to Table \ref{tab1}, our result is an improvement of Mubarakzyanov's ones (cf. \cite[Part 4]{SW14} for more details). 
   For instance, we have removed $g_{5,24}^{-1}$ and $g_{5,26}^{-1p}$ (this also can be seen because 
   $g_{5,24}^{-1} \cong g_{5,24}^1$ and $g_{5,26}^{-1p} \cong g_{5,26}^{1p}$ by changing the sign of $X_4$).
\end{rem}

\section{The problem of classifying \Lntwo}\label{sec4}

The goal of this section is to prove Theorem \ref{Thm2}. Before giving the proof of Theorem \ref{Thm2}, 
we first explore some properties of {\Lntwo}. Recall that {\Lntwo} is the class of all $(n+2)$-dimensional 
real solvable Lie algebras having $n$-dimensional derived algebras.

\subsection{Description of \Lntwo}\label{subsec4.1}

To classify {\Lntwo}, we proceed in a similar way as described in Subsection \ref{subsec3.1}:
\begin{enumerate}[Step 1.]
	\item Take an $n$-dimensional real nilpotent Lie algebra $H$;
	\item\label{step2-Lie(n+2,n)}  Extend $H$ to all $(n+2)$-dimensional real solvable Lie algebras $L$ 
	which admits $H$ as derived algebra, i.e. $L^1 = H$. Afterward, classify such Lie algebras;
	\item\label{step3-Lie(n+2,n)}  Repeat two steps above for all possibilities of $H$.
\end{enumerate}

First of all, we fix an arbitrary $n$-dimensional real nilpotent Lie algebra $H$. Let $L \in \li$ with $L^1 = H$. 
Without loss of generality, we can assume
\[
	\begin{array}{l l}
		L = \s \{x_1, \ldots, x_n, y, z\}, & L^1 = H = \s \{x_1, \ldots, x_n\}.
	\end{array}
\]
By Proposition \ref{Lem2.1-deGraaf}, $L$ can be represented in the following form:
\[
	\begin{array}{l l}
		L = \R z \oplus_d K = \R z \oplus_d \left(\R y \oplus_{d'} H\right), & d \in \D(K), d' \in \D(H).
	\end{array}
\]

However, it is not true that all Lie algebras of the above forms belong to \Lntwo. 
We give a necessary and sufficient condition of the pair $(d, d')$ for which $L \in \li$.

\begin{prop}\label{Pro4.1}
	Let $H$ be an $n$-dimensional real nilpotent Lie algebra, and 
	$L = \R z \oplus_d K = \R z \oplus_d \left(\R y \oplus_{d'} H\right)$ with $d \in \D(K)$, 
	$d(K) \subset H$ and $d' \in \D(H)$ as above. By renumbering, if necessary, 
	we can assume that $H^1 = \s \{x_1, \ldots, x_m\}$  for some $0 \leq m < n$. 
	Then, the following conditions are equivalent:
	\begin{enumerate}
		\item $L \in \li$;
		\item $L^1 = H$;
		\item $\s \{x_{m+1}, \ldots, x_n\} \subset d(K) + d'(H) + H^1$;
		\item\label{part4-Pro4.1} $H/H^1 = \I \td + \I \td'$, where $\td \colon K/H^1 \to K/H^1$ 
		and $\td' \colon H/H^1 \to H/H^1$ are homomorphisms induced from $d$ and $d'$ 
		on the quotient Lie algebras $K/H^1$ and $H/H^1$, respectively.
	\end{enumerate}
\end{prop}

\begin{proof}
	First of all, we note that:
	\begin{flalign*}
		L^1 & = \s \{[z, u] \big| u \in K\} + \s \{[y, v] \big| v \in H\} + \s \{[u', v'] \big| u', v' \in H\} \\
		& = d(K) + d'(H) + H^1 \subset H \quad \left(\text{because $d(K) \subset H$}\right).
	\end{flalign*}
	Therefore, $L \in \li$ if and only if
	\[
		\begin{array}{l l}
			& L^1 = H \\
			\Leftrightarrow & d(K) + d'(H) + H^1 = H \\
			\Leftrightarrow & d(K) + d'(H) + H^1 = \s\{x_{m+1}, \ldots, x_n\} + H^1 \\
			\Leftrightarrow & \s \{x_{m+1}, \ldots, x_n\} \subset d(K) + d'(H) + H^1 \\
		    & \hfill \left(\text{since $\s \{x_1, \ldots, x_m\} = H^1$}\right)
		\end{array}
	\]
	By the Leibniz formula, the derived algebra $H^1$ is invariant with respect to $d$ and $d'$, 
	i.e. $d(H^1), d'(H^1) \subset H^1$. Therefore, $\td \colon K/H^1 \to K/H^1$ and $\td' \colon H/H^1 \to H/H^1$ 
	are well-defined, and the last equation above means that $H/H^1 = \I \td + \I \td'$. 
	This complete the proof of Proposition \ref{Pro4.1}.
\end{proof}

Thus, each Lie algebra $L = \R z \oplus_{d} \left(\R y \oplus_{d'} H\right) \in \li$ with $L^1 = H$ 
can be seen as an extension of $H$ by a pair of derivations $(d,d')$ which satisfies Proposition \ref{Pro4.1}. 
Furthermore, an immediate consequence of Proposition \ref{Pro4.1} and Corollary \ref{Cor2.10} is as follows.

\begin{cor}\label{Cor4.2}
	If $L = \R z \oplus_d \left(\R y \oplus_{d'} H\right)$ with $d \in \ad(\R y \oplus_{d'} H)$ and $d' \in \ad(H)$ then $L \notin \li$. 
	In other words, all  Lie algebras extended from $H$ by a pair of inner derivations cannot belong to \Lntwo.
\end{cor}

\begin{prop}\label{Pro4.3}
	An $L \in \li$ is decomposable if and only if there exist $L_1 \in \mathrm{Lie} \left(m_1 + 1, m_1\right)$ 
	and $L_2 \in \mathrm{Lie} \left(m_2 + 1, m_2\right)$ such that $L = L_1 \oplus L_2$ for some 
	$m_1, m_2 \geq 0$ and $m_1 + m_2 = n$.
\end{prop}

\begin{proof}
	Assume that $L$ is decomposable, i.e. $L = L_1 \oplus L_2$ for certain two real solvable Lie algebras 
	$L_1$ and $L_2$. In particular, $\dim L_1 + \dim L_2 = \dim L = n + 2$. For convenience, we set 
	\[
		\dim L_1 = m_1 + 1, \quad \dim L_2 = m_2 +1,
	\]
	with $m_1, m_2 \geq 0$ and $m_1 + m_2 = n$. It follows directly from the solvability of $L_1, L_2$ that
	\[
		\begin{cases}
			\dim L_1^1 \leq \dim L_1 - 1 = m_1, \\ \dim L_2^1 \leq \dim L_2 - 1 = m_2.
		\end{cases}
	\]
	Besides, $L^1 = L_1^1 \oplus L_2^2$ implies
	\[
		n = \dim L^1 = \dim L_1^1 + \dim L_2^1 \leq m_1 + m_2 = n.
	\]
	Equality holds if $\dim L_1^1 = m_1$ and $\dim L_2^1 = m_2$ which lead to
	\[
		L_1 \in \mathrm{Lie} \left(m_1 + 1, m_1\right), \quad L_2 \in \mathrm{Lie} \left(m_2 +1, m_2\right).
	\]
	The converse is straightforward. The proof is complete.
\end{proof}

Thanks to Proposition~\ref{Pro4.3}, we only need to pay attention to indecomposable Lie algebras.

\begin{prop}\label{Pro4.4}
	Let $H$ be an $n$-dimensional real nilpotent Lie algebra and $L \in \li$ with $L^1 = H$. 
	Then $L$ is decomposable if and only if there exists a basis $\{x_1, \ldots, x_n, y, z\}$ of $L$ such that 
	$L = \R z \oplus_d \left(\R y \oplus_{d'} H\right)$, $[z, y] = 0$ and the pair $(d,d')$ satisfy the following condition: 
	there exist Lie algebras $H_1, H_2 \subseteq H$ such that
	\[
		\begin{array}{l l}
			\begin{cases}
				d(H_1) \subset H_1, & d(H_2) = 0, \\
				d'(H_1) = 0, & d'(H_2) \subset H_2,
			\end{cases} &
			\text{and $H_1 \oplus H_2 = H$}.
		\end{array}
	\]
	In particular, if $H$ is indecomposable then $L$ is decomposable if and only if 
	$L \cong \R \oplus \bl$ with $\bl \in \Li$ and $\bl^1 = L^1 = H$.
\end{prop}

\begin{proof}
	\begin{enumerate}
		\item[$(\Leftarrow)$] It is obvious. More precisely, we have
		\[
			L = \left(\R z \oplus_d H_1\right) \oplus \left(\R y \oplus_{d'} H_2\right).
		\]
		
		\item[$(\Rightarrow)$] According to Proposition \ref{Pro4.3}, there exist $L_1 \in \mathrm{Lie}(m_1+1, m_1)$ 
		and $L_2 \in \mathrm{Lie}(m_2+1, m_2)$ such that
		\[
			\begin{array}{l l}
				L = L_1 \oplus L_2, & m_1, m_2 \geq 0, m_1 + m_2 = n.
			\end{array}
		\]
		Then $H = L_1^1 \oplus L_2^1 \co H_1 \oplus H_2$. Assume that
		\[
			\begin{array}{l l}
				H_1 = \s \{x_1, \ldots, x_{m_1}\}, & H_2 = \s \{x_{m_1+1}, \ldots, x_n\}.
			\end{array}
		\]
		We can always supplement $z$ to $\{x_1, \ldots, x_{m_1}\}$ and $y$ to $\{x_{m_1+1}, \ldots, x_n\}$ 
		to get bases of $L_1$ and $L_2$, respectively. By this way, we have $[z, y] = 0$, $H = \s \{x_1, \ldots, x_n\}$ 
		and $L = \R z \oplus_d \left(\R y \oplus_{d'} H\right)$ in which
		\[
			\begin{cases}
				[z, H_1] = d(H_1) \subset H_1, & [z, H_2] = d(H_2) = 0, \\
				[y, H_1] = d'(H_1) = 0, & [y, H_2] = d'(H_2) \subset H_2.
			\end{cases}
		\]
	\end{enumerate}
	\noindent The proof of Proposition \ref{Pro4.4} is complete.
\end{proof}

\subsection{Proof of Theorem \ref{Thm2}}\label{subsec4.2}

We will show that the problem of classifying {\Lntwo} contains a wild problem. 
In fact, let us consider the following class
\[
	\lic \co \left\lbrace L \in \li \, \big| \, L^1 = \R^n \right\rbrace \subset \li.
\]
We will prove that the problem of classifying {\Lntwoc} is wild.

Let $L = \s \{x_1, \ldots, x_n, y, z\} \in \lic$ such that 
\[
	L^1 = \s \{x_1, \ldots, x_n\} = \R^n.
\] 
By Proposition \ref{Lem2.1-deGraaf}, we represent $L$ in the following form: 
\[
L = \R z \oplus_d K = \R z \oplus_d \left(\R y \oplus_{d'} \R^n\right),
\]	
where $d \in \D(K)$ and $d' = \ad_y|_{\R^n} = a_y \in \D(\R^n)$ satisfy equivalent conditions in 
Proposition \ref{Pro4.1}. For simplicity, we assume additionally two conditions as follows:
\begin{itemize}
	\item $[z, y] = 0$ which allows us to identify
	\[
		d = \ad_z|_{\R^n} \oplus 0 = a_z \oplus 0 \equiv a_z,
	\]
	and consider the pair $(d, d')$ as derivations of $\R^n$.
	\item $d$ and $d'$ are outer derivations of $\R^n$. Moreover, they must be non-proportional, 
	in order to guarantee that $L$ is indecomposable.
\end{itemize} 
Even if we treat this simpler case, the result is as follows.

\begin{prop}\label{Pro4.5}
	Let $L_i = \R z \oplus_{d_i} \left(\R y \oplus_{d'_i} \R^n \right)$ for $i = 1, 2$, 
	be two Lie algebras in {\Lntwoc} which satisfy all of the above conditions. 
	Then $L_1 \cong L_2$ if and only if there exists $\sigma \in \Au(\R^n)$ and $\begin{bmatrix} \alpha & \beta \\ \gamma & \delta \end{bmatrix} \in \GL_2 (\R)$ such that
	\begin{equation}\label{Cond4.1}
		\begin{cases}
			\sigma d_1 \sigma^{-1} = \gamma d_2 + \alpha d_2' \\ \sigma d'_1 \sigma^{-1} = \delta d_2 + \beta d_2'  
		\end{cases}
	\end{equation}
\end{prop}

\begin{proof}
	\begin{enumerate}
		\item[$\left(\Rightarrow\right)$] If $\ts \colon L_1 \to L_2$ is an isomorphism, then so is $\sigma \co \ts|_{\R^n}$. 
		Setting
		\[
			\begin{cases}
				\ts(z) = \gamma z + \alpha y + u, & u \in \R^n, \\
				\ts(y) = \delta z + \beta y + v, & v \in \R^n. 
			\end{cases}
		\]
		Since $d_i$ and $d'_i$ ($i =1, 2$) are non-proportional, $\det \begin{bmatrix} \alpha & \beta \\ \gamma & \delta \end{bmatrix} \neq 0$, i.e. $\begin{bmatrix} \alpha & \beta \\ \gamma & \delta \end{bmatrix} \in \GL_2 (\R)$. 
		Now, for arbitrary $x \in \R^n$, we have:
		\[
			\begin{array}{l l}
				& \ts([z, x]) = \left[\ts(z), \ts(x)\right] \\
				\Leftrightarrow & \sigma([z, x])=[\gamma z + \alpha y + u, \sigma(x)] \\
				\Leftrightarrow & \sigma([z, x])=[\gamma z + \alpha y, \sigma(x)] \\
				\Leftrightarrow & \sigma d_1(x) = (\gamma d_2 + \alpha d'_2)\sigma(x).
			\end{array}
		\]
		Thus, $\sigma d_1 \sigma^{-1} = \gamma d_2 + \alpha d'_2$.
		Similarly, replacing $z$ by $y$ we get $\sigma d'_1 \sigma^{-1} = \delta d_2 + \beta d'_2$.
	
		\item[$(\Leftarrow)$] Assume that there exists $\sigma \in \Au ({\R}^n)$ and $\begin{bmatrix} \alpha & \beta \\ \gamma & \delta \end{bmatrix} \in \GL_2 (\R)$ which satisfy Condition \eqref{Cond4.1}. 
		We define $\ts \colon L_1 \to L_2$ as follows:
		\[
			\begin{cases}
				\ts (x) = \sigma (x), & x \in \R^n, \\
				\ts(z) = \gamma z + \alpha y, \\
				\ts(y) = \delta z + \beta y.
			\end{cases}
		\]
		Since $\det \begin{bmatrix} \alpha & \beta \\ \gamma & \delta \end{bmatrix} \neq 0$, $\ts$ 
		is a linear isomorphism. Moreover, it also preserves Lie brackets. In fact, the equation 
		$\sigma d_1 \sigma^{-1} = \gamma d_2 + \alpha d_2'$ (resp. $\sigma d'_1\sigma^{-1} = \delta d_2 + \beta d'_2$)
		implies $\ts([z, x]) = \left[\ts(z), \ts(x)\right]$ (resp. $\ts([y, x]) = \left[\ts(y), \ts(x)\right]$) for every $x \in \R^n$. 
		Besides, $\ts([z, y]) = \left[\ts(z), \ts(y)\right]$ is equivalent to 
		$[z, y] \det \begin{bmatrix} \alpha & \beta \\ \gamma & \delta \end{bmatrix} = 0$ which is obviously true 
		because of the assumption $[z, y] = 0$.	Thus $\ts$ is an isomorphism and $L_1 \cong L_2$.
	\end{enumerate}
	\noindent The proof of Proposition \ref{Pro4.5} is complete.
\end{proof}

\begin{pot2}
	According to Definition \ref{Definition2.4}, two pairs $\left(d_1, d'_1\right)$ and $\left(d_2, d'_2\right)$ 
	satisfied Condition \eqref{Cond4.1} in Proposition  \ref{Pro4.5} are weakly similar. By Remark \ref{Remark2.5}, 
	the problem of classifying pairs of matrices up to weak similarity, even if the pairs of commuting matrices, is wild. 
	That means the problem of classifying {\Lntwoc} is wild. So is the problem of classifying \Lntwo.\hfill{$\square$}
\end{pot2}

\begin{rem}\label{Rem4.6}
	The wildness of the problem of classifying {\Lntwo} is slightly different from that of {\Lnone} because it is wild 
	not only in Step \ref{step2-Lie(n+2,n)} but also in Step \ref{step3-Lie(n+2,n)} in Subsection \ref{subsec4.1}.
\end{rem}

\section{A special case of \Lntwo}\label{sec5}
   
This section is devoted to consider a special case of {\Lntwo}.
	
Let $H$ be an $n$-dimensional real nilpotent Lie algebra. By Subsection \ref{subsec4.1}, 
all $L \in \li$ with $L^1 = H$ are of the following form
\[
	L = \R z \oplus_d K = \R z \oplus_d \left(\R y \oplus_{d'} H\right),
\]
where $d \in \D(K)$ and $d' \in \D(H)$ satisfy equivalent conditions in Proposition \ref{Pro4.1}. 
As we have seen in Subsection \ref{subsec4.2}, if the pair $(d,d')$ consists of outer derivations 
then the classification problem is wild. Therefore, it is natural to consider a subclass {\Lntwoa} 
consists of all Lie algebras of the forms:
\[
	L = \R z \oplus_d \left(\R y \oplus_{d'} H\right),
\]
where $H$ is an arbitrary $n$-dimensional real nilpotent Lie algebra, and the following conditions hold:
\begin{itemize}
	\item $(d, d')$ satisfies the equivalent conditions in Proposition \ref{Pro4.1}, or equivalently, $L$ belongs to \Lntwo;
	\item $(d, d')$ contains at least one inner derivation.
\end{itemize}

\subsection{Proof of Theorem \ref{Thm3}}

The proof of Theorem \ref{Thm3} will take place through certain steps in which we need some support results as follows.

\begin{prop}\label{Pro5.1}
	For any indecomposable Lie algebra $L \in \lia$ with $L^1 = H$, there exist $y, z \in L \setminus H$ 
	and $d \in \D(K) \setminus \ad(K)$ such that $L = \R z \oplus_d (\R y \oplus H)$ where $K \co \R y \oplus H$.	
\end{prop}

\begin{proof}
	Since $L \in \lia$ with $L^1 = H$, there exist $y', z \in L \setminus H$ as well as $d' \in \D(H)$ 
	and $d \in \D(\R y' \oplus_{d'} H)$ such that  
	\[
		L = \R z \oplus_d \left(\R y' \oplus_{d'} H\right).
	\]
	Set $K \co \R y' \oplus_{d'} H$. First of all, we note that
	\begin{itemize}
		\item By Corollary \ref{Cor4.2}, $d$ and $d'$ cannot be inner derivations simultaneously;
		\item If $d \in \ad(K)$ then $L$ is decomposable by Corollary \ref{Cor2.10} which 
		conflicts with the indecomposability of $L$.
	\end{itemize}
	Therefore, the pair $(d,d')$ contains one and only one inner derivation which is exactly $d'$, 
	i.e. $d' \in \ad(H)$. Taking account of Corollary \ref{Cor2.10}, we have
	\[
		\begin{array}{l l}
			K= \R y' \oplus_{d'} H = \R y \oplus H, & \text{for some $y \in K \setminus H$}.
		\end{array}
	\]
	Thus $L = \R z \oplus_d K = \R z \oplus_d \left(\R y \oplus H\right)$. The proof is complete.
\end{proof}

According to Proposition \ref{Pro5.1}, to classify {\Lntwoa}, we need to point out conditions 
of $d_1$ and $d_2$ such that two Lie algebras
\[
	\begin{array}{l l}
		L_i = \R z \oplus_{d_i} K = \R z \oplus_{d_i} (\R y \oplus H), &  i = 1, 2,
	\end{array}
\]
determined by $d_1, d_2 \in \D(K) \setminus \ad(K)$ are isomorphic. 
To this end, we next explore some additional properties of $d$.

\begin{prop}\label{Pro5.2}
	If $L = \R z \oplus_d (\R y \oplus H) \in \lia$ then $d|_H = a_z$ is an outer derivation of $H$.
\end{prop}

\begin{proof}
	Assume that $d|_H = a_z = \ad_u$ with $u \in H$. By changing $z' = z - u$ we have $[z', x] = 0$ for all $x \in H$. 
	Therefore, without loss of generality, we can assume that $[z, H] = 0$. 
	By this way, all Lie brackets of $L$ are determined by the original ones of $H$ and $[z, y] = d(y) \in H$. 
	Taking account of the fact $L^1 = H$, this means that
	\begin{equation}\label{eq5.1}
		H = H^1 + \s \{[z, y]\}.
	\end{equation}
	Since $\dim H = n$ and $\dim H^1 <n$, it implies $\dim H^1 = n - 1$ and $[z, y] \notin H^1$. 
	Therefore, $H \in \mathrm{Lie}(n,n - 1)$ and $[z, y] \in H \setminus H^1$.
	
	On the other hand, it follows from the Jacobi identity that
	\[
		  [[z, y], x] = [[x, y], z] + [[z, x], y] = 0, \quad \text{for all $x \in H$},
	\]
	i.e. $[z, y] \in \Z(H)$, and equation \eqref{eq5.1} becomes to
	\[
		H = H^1 \oplus \s \{[z, y]\}
	\]
	which conflicts with $H \in \mathrm{Lie}(n,n - 1)$ and Proposition \ref{Pro3.2}. 
	So $d|_H = a_z$ cannot be an inner derivation of $H$. The proof is complete.
\end{proof}

\begin{prop}\label{Pro5.3}
	Let $L = \R z \oplus_d (\R y \oplus H) \in \lia$. Then the following assertions are equivalent:
	\begin{enumerate}
		\item $L$ is decomposable;
		\item $L \cong \R \oplus \bl$, where $\bl \in \Li$ and $\bl^1 = L^1 = H$;
		\item $[z, y] \in d(\Z(H))$.
	\end{enumerate}
	In particular, if $H = \R^n$ then $L$ is decomposable if and only if $d|_{\R^n} = a_z$ is nonsingular.
\end{prop}

\begin{proof}
	We will prove that $1 \Rightarrow 2 \Rightarrow 3 \Rightarrow 1$.
	\begin{enumerate}
		\item[] \hspace{-1.2cm} $\left(1 \Rightarrow 2\right)$ Assume that $L = L_1 \oplus L_2$. 
		Then $H = H_1 \oplus H_2 \co L_1^1 \oplus L_2^1$, where $L_1 \in \mathrm{Lie}(m_1+1,m_1)$, 
		$L_2 \in \mathrm{Lie}(m_2+1,m_2)$ and $m_1 + m_2 = n$. 
		We assert that the case in which $m_1, m_2 \geq 1$ cannot happen.
		
		In fact, in that case, without loss of generality, we can assume that
		\[
			\begin{array}{l l}
				H_1 = \s \{x_1, \ldots, x_{m_1}\}, & H_2 = \s \{x_{m_1+1}, \ldots, x_n\}.
			\end{array}
		\]
		Then we can always supplement two elements $z', y' \in L \setminus H$ such that
		\[
			L = \left(\R z' \oplus_D H_1\right) \oplus \left(\R y' \oplus_{D'} H_2\right).
		\]
		Because $L_1 \in \mathrm{Lie}(m_1+1,m_1)$, there exists $x_{i_0} \in H_1 $ such that $[z', x_{i_0}] \notin H_1^1$ 	
		(similarly, we can also choose $x_{i_0} \in H_2 $ such that $ [y', x_{i_0}] \notin H_2^1$). Put
		\[
			\begin{cases}
				z' = a_1z + b_1y + u_1, & u_1 \in H, \\
				y' = a_2z + b_2y + u_2, & u_2 \in H.
			\end{cases}
		\]
		Note that we must have $a_1a_2 \neq 0$ since $L_1 \in \mathrm{Lie}(m_1+1,m_1)$ and $L_2 \in \mathrm{Lie}(m_2+1,m_2)$. Then
		\[
			[z, x_{i_0}] = \frac{1}{a_1} \left([z', x_{i_0}] - [u_1, x_{i_0}]\right) \notin H_1^1
		\]
		which implies $[y', x_{i_0}] = a_2 [z, x_{i_0}] + [u_2, x_{i_0}] \neq 0$. This is a contradiction. 
		
		The above contradictions show that $m_1 = 0$ or $m_2 = 0$, i.e. $L \cong \R \oplus \bl$. 
		It is obvious that $\bl \in \Li$ and $\bl^1 = L^1 = H$.

		\item[] \hspace{-1.2cm} $\left(2 \Rightarrow 3\right)$ Assume that $L = \R z' \oplus \bl \co \R z' \oplus \left(\R y' \oplus_{d'} H\right)$ 
		with $\bl \in \Li$ and $\bl^1 = L^1 = H$. Put
		\[
			\begin{cases}
				z' = \alpha_1z + \beta_1y + v_1, & v_1 \in H, \\
				y' = \alpha_2z + \beta_2y + v_2, & v_2 \in H.
			\end{cases}
		\]
		Then $\alpha_2 \neq 0$ since on the contrary, it conflicts with $\bl \in \Li$. 
		
		By similar arguments as above, there exists $x_{i_0} \in H$ such that $[y', x_{i_0}] \notin H^1$. 
		So $[z, x_{i_0}] = \frac{1}{\alpha_2} \left([y', x_{i_0}] - [v_2, x_{i_0}]\right) \notin H^1$. Then
		\[
			\begin{array}{l l l}
				0 = [z', x_{i_0}] = \alpha_1[z, x_{i_0}] + [v_1, x_{i_0}]
				& \Leftrightarrow & \begin{cases} \alpha_1 = 0 \\ [v_1, x_{i_0}] = 0. \end{cases}
			\end{array}
		\]
		This implies that $[v_1, x] = [z', x] = 0$ for all $x \in H$, i.e. $v_1 \in \Z(H)$.
		
		Now, $\beta_1\alpha_2 \neq 0$ because $z' \in \bl$ if $\beta_1\alpha_2 = 0$. Therefore, we have
		\[
			0 = [z', y'] = [\beta_1y + v_1, \alpha_2z + \beta_2y + v_2] = -\beta_1\alpha_2[z, y] - \alpha_2[z, v_1]
		\]
		which leads to $[z, y] = -\frac{1}{\beta_1}[z, v_1] = -\frac{1}{\beta_1} d(v_1) \in d(\Z(H))$.
		
		\item[] \hspace{-1.2cm} $\left(3 \Rightarrow 1\right)$ If $[z, y] = x \in d(\Z(H))$ then there exists $x' \in \Z(H)$ 
		such that $d(x') = [z, x'] = x$. By changing $y' = y - x'$ we have $[z, y'] = 0$ and $[y', u] = 0$ for all $u \in H$. 
		This means that
		\[
			L = \R y' \oplus \s \{x_1, \ldots, x_n, z\} \co \R y' \oplus \bl,
		\]
		i.e. $L$ is decomposable.
	\end{enumerate}
	Finally, if $H = \R^n$ the $L$ is decomposable if and only if $[z, y] \in d(\Z(\R^n)) = d(\R^n)$. 
	By Proposition \ref{Pro4.1}, we have
	\[
		\begin{array}{l l l}
			L \in \lia & \Leftrightarrow & d(\R y \oplus \R^n) = \R^n \\
			& \Leftrightarrow & \s \{[z, y]\} + d(\R^n) = \R^n \\
			& \Leftrightarrow & d(\R^n) = \R^n \quad (\text{since } [z, y] \in d(\R^n)) \\
			& \Leftrightarrow & \text{$d|_{\R^n} = a_z$ is nonsingular}.
		\end{array}
	\]
	The proof of Proposition \ref{Pro5.3} is complete.
\end{proof}

Now, we formulate the desired isomorphic condition on {\Lntwoa} in Proposition \ref{Pro5.4} below. 

\begin{prop}\label{Pro5.4}
	Let
	\[
		\begin{array}{l l}
			L_i = \R z \oplus_{d_i} K = \R z \oplus_{d_i} (\R y \oplus H) \in \lia, & i = 1, 2.
		\end{array}
	\]
	Then $L_1 \cong L_2$ if and only if $\bd_1 \sim_p \bd_2$, where $\bd_i$ are the equivalent classes of $d_i \in \D(K)/\ad(K) = \hk$.
\end{prop}

\begin{proof}
	\begin{enumerate}
		\item[$\left(\Leftarrow\right)$] A direct application of Corollary \ref{Cor2.9}.
	
		\item[$\left(\Rightarrow\right)$] Let $\ts \colon \R z \oplus_{d_1} K \to \R z \oplus_{d_2} K$ be an isomorphism. 
		We claim that $K$ is $\ts$-invariant, i.e. $\ts(K) = K$. Indeed, set $\ts(y) = \alpha z + \beta y + u$ for $u \in H$. 
		Then for arbitrary $x \in H$, we have
		\[
			\begin{array}{l l}
				& \ts \left([y, \ts^{-1} (x)]\right) = [\ts(y), x] \\
				\Leftrightarrow & 0 = [\alpha z + \beta y + u, x] = \alpha [z, x ] + [u, x] \\
				\Leftrightarrow & \alpha a_z (x) = -\ad_u (x) \\
				\Leftrightarrow & \alpha = 0 \quad \text{(because $a_z$ is an outer derivation of $H$)}.
			\end{array}
		\]
		Since $K$ is $\ts$-invariant, $\sigma \co \ts|_K \colon K \to K$ is well-defined. 
		It is obvious that $\sigma$ is also an isomorphism. Finally, we set $\ts(z) = \alpha'z + v$ 
		where $\alpha' \neq 0$ and $v \in K$. Then for every $x \in K$, we have
		\[
			\begin{array}{l l}
				& \ts ([z, x]) = \left[\ts(z), \ts(x)\right] \\
				\Leftrightarrow & \sigma ([z, x]) = [\alpha'z + v, \sigma(x)] \\
				\Leftrightarrow & \sigma d_1(x) = (\alpha'd_2 + \ad_v) \sigma(x) \\
				\Leftrightarrow & \sigma d_1 \sigma^{-1} = \alpha'd_2 + \ad_v.
			\end{array}
		\]
		This means $\bd_1 \sim_p \bd_2$ in $\D (K)/\ad(K) = \hk$.	
	\end{enumerate}
	The proof of Proposition \ref{Pro5.4} is complete.
\end{proof}

\begin{pot3}
	As a direct consequence of Propositions \ref{Pro5.1}, \ref{Pro5.2} and \ref{Pro5.4}, 
	the problem of classifying {\Lntwoa} is equivalent to the problem of classifying equivalent 
	classes in $\D(K)/\ad(K) = \hk$ of outer derivations of $K$ which satisfy equivalent 
	conditions in Proposition \ref{Pro4.1} up to proportional similarity, where $H$ is an arbitrary 
	$n$-dimensional real nilpotent Lie algebra and $K = \R \oplus H$.\hfill{$\square$}
\end{pot3}

\begin{rem}
	The problem of classifying {\Lntwoa} is essentially wild since it also requires the classification 
	of real nilpotent Lie algebras. However, we can determine a full classification of {\Lntwoa} 
	whenever a classification of $n$-dimensional real nilpotent Lie algebras is given.
\end{rem}

\subsection{Illustrative examples}

In the remaining of this section, we give classifications of {\Lntwoa} in low dimensions 
by using the technique proposed in the previous subsection. More concretely, we classify 
{\Lntwoa} for $n = 2$ in Example \ref{Ex5.6} and for $n = 3$ in Example \ref{Ex5.7}. 
In view of Proposition \ref{Pro4.3}, all decomposable Lie algebras in {\Lntwo} are directly reduced to 
$\mathrm{Lie} (m+1, m)$ for some $0 < m \leq n$. Therefore, we only pay attention to indecomposable ones.

\begin{ex}[The classification of $\mathrm{Lie}_\ad (4,2)$]\label{Ex5.6}
	Let $L \in \mathrm{Lie}_\ad(4,2)$. Since $L^1 = \R^2 = \s \{x_1, x_2\}$, we set
	\[
		L = L_d \co \R x_4 \oplus_d K = \R x_4 \oplus_d \left(\R x_3 \oplus \R^2\right)
	\]
	where $d \in \D(K) \setminus \ad(K)$ has the following form
	\[
		d = \begin{bmatrix} a_1 & b_1 & c_1 \\ a_2 & b_2 & c_2 \\ 0 & 0 & 0 \end{bmatrix} \neq 0.
	\]
	By Condition \ref{part4-Pro4.1} of Proposition \ref{Pro4.1}, $L_d \in \mathrm{Lie}_\ad (4,2)$ if and only if $\ra d = 2$. 
	To ensure that $L_d$ is indecomposable, it follows from Proposition \ref{Pro5.3} that $a_{x_4} = d|_{\R^2}$ is singular, 
	i.e. $\det \begin{bmatrix} a_1 & b_1 \\ a_2 & b_2 \end{bmatrix} = 0$. 
	The possible JCFs of $d$ with these conditions are as follows:
	\[
		\begin{array}{l l l}
			\begin{bmatrix} \lambda \\ & 0 & 1 \\ & & 0 \end{bmatrix} (\lambda \neq 0)
			\sim_p \begin{bmatrix} 1 \\ & 0 & 1 \\ & & 0 \end{bmatrix} \co A& \text{or} &
			\begin{bmatrix} 0 & 1 \\  & 0 & 1 \\ & & 0 \end{bmatrix} \co B.
		\end{array}
	\]
	By Proposition \ref{Pro5.4}, we have two Lie algebras $L_A$ and $L_B$ in $\mathrm{Lie}_\ad (4,2)$. 
	In particular, these results concise to that of \cite[Theorem 10, Part 1]{VTTTT19}, 
	namely, $L_A \cong \mathcal{G}_{4,1}$ and $L_B \cong \mathcal{G}_{4,2}$.
\end{ex}

\begin{ex}[The classification of $\mathrm{Lie}_\ad(5, 3)$]\label{Ex5.7}
	Let $L \in \mathrm{Lie}_\ad(5,3)$. As we have known, $L^1 = \R^3$ or $L^1 = \h_3$.
	\begin{enumerate}[\bf A.]
		\item Let $H = \R^3 = \s \{x_1, x_2, x_3\}$. In this case, we set
		\[
			L = L_d \co \R x_5 \oplus_d K = \R x_5 \oplus_d \left(\R x_4 \oplus \R^3\right)
		\]
		where $d \in \D(K) \setminus \ad(K)$ has the following form
		\[
			d = \begin{bmatrix} a_1 & b_1 & c_1 & e_1 \\ a_2 & b_2 & c_2 & e_2 \\ a_3 & b_3 & c_3 & e_3 \\ 0 & 0 & 0 & 0 \end{bmatrix} \neq 0.
		\]
		Similarly, by Condition \ref{part4-Pro4.1} of Proposition \ref{Pro4.1}, we must have $\ra d = 3$. 
		To ensure that $L_d$ is indecomposable, it follows from Proposition \ref{Pro5.3} that $a_{x_5} = d|_{\R^3}$ is singular, 
		i.e. $\det \begin{bmatrix} a_1 & b_1 & c_1 \\ a_2 & b_2 & c_2 \\ a_3 & b_3 & c_3 \end{bmatrix} = 0$. 
		The possible JCFs of $d$ with these conditions are as follows:
		\[
			\begin{array}{l l l}
				\begin{bmatrix} \lambda_1 \\ & \lambda_2 \\ & & 0 & 1 \\ & & & 0 \end{bmatrix} (0 < |\lambda_2| \leq |\lambda_1|)
					\sim_p \begin{bmatrix} 1 \\ & \lambda \\ & & 0 & 1 \\ & & & 0 \end{bmatrix} \co A(\lambda) \; (0 < |\lambda| \leq 1), \\
				\begin{bmatrix} \lambda & 1 \\  & \lambda \\ & & 0 & 1 \\ & & & 0 \end{bmatrix} (\lambda \neq 0) 
					\sim_p \begin{bmatrix} 1 & 1 \\ & 1 \\ & & 0 & 1 \\ & & & 0 \end{bmatrix} \co B, \\
				\begin{bmatrix} \lambda \\ & 0 & 1 \\ & & 0 & 1 \\ & & & 0 \end{bmatrix} (\lambda \neq 0) 
					\sim_p \begin{bmatrix} 1 \\ & 0 & 1 \\ & & 0 & 1 \\ & & & 0 \end{bmatrix} \co C, \\
				\begin{bmatrix} 0 & 1 \\ & 0 & 1 \\ & & 0 & 1 \\ & & & 0 \end{bmatrix} \co D, \\
			\end{array}
		\]
		
				\[
		\begin{array}{l l l}
		\begin{bmatrix} \lambda & 1 \\ -1 & \lambda \\ & & 0 & 1 \\ & & & 0 \end{bmatrix} (\lambda \neq 0)
		\sim_p \begin{bmatrix} |\lambda| & 1 \\ -1 & |\lambda| \\ & & 0 & 1 \\ & & & 0 \end{bmatrix} \co E(\lambda) \; (\lambda \geq 0). \\
		\end{array}
		\]
		
		Finally, it follows from Proposition \ref{Pro5.4} that we have five families of Lie algebras in 
		$\mathrm{Lie}_\ad (5,3)$ with derived algebra $\R^3$, namely, $L_{A(\lambda)}$, 
		$L_B$, $L_C$, $L_D$ and $L_{E(\lambda)}$.
		
		\item Let $H = \h_3 = \s \left\lbrace x_1, x_2, x_3 \, \big| \, [x_2, x_3] = x_1\right\rbrace$. In this case, we set
		\[
			L = L_d \co \R x_5 \oplus_d K = \R x_5 \oplus_d  \left(\R x_4 \oplus H\right),
		\]
		where $d \in \D(K) \setminus \ad(K)$ and $\td \colon K/H^1 \to K/H^1$ induced from $d$ are of the following forms:
		\[
			\begin{array}{l l}
				d = \begin{bmatrix} a+b & f & g & h \\ 0 & a & c & 0 \\ 0 & e & b & 0 \\ 0 & 0 & 0 & 0 \end{bmatrix}, &
				\td = \begin{bmatrix} a & c & 0 \\ e & b & 0 \\ 0 & 0 & 0 \end{bmatrix}.
			\end{array}
		\]
		First of all, we have
		\[
			\hk = \left\lbrace \bd = \begin{bmatrix} a+b & 0 & 0 & h \\ 0 & a & c & 0 \\ 0 & e & b & 0 \\ 0 & 0 & 0 & 0 \end{bmatrix}
			\Bigg| \, a, b, c, e, h \in \R \right\rbrace.
		\]
		Next, Condition \ref{part4-Pro4.1} of Proposition \ref{Pro4.1} implies
		\[
			\begin{array}{l l l l l}
				L_d \in \mathrm{Lie}_\ad (5,3) & \Leftrightarrow & \ra \td = 2 &
				\Leftrightarrow & \det \begin{bmatrix} a & c \\ e & b \end{bmatrix} \neq 0.
			\end{array}
		\]
		Now, $d(\Z(H)) = \s \{(a+b)x_1\}$ since $\Z(H) = \s \{x_1\}$. 
		It follows from Proposition \ref{Pro5.3} that $L_d$ is indecomposable if and only if
		\[
			\begin{array}{l l l}
				[z, y] = hx_1 \notin \s \{(a+b)x_1\}	& \Leftrightarrow & h \neq 0 = a+b.
			\end{array}
		\]
		The possible JCFs of $\bd \in \hk$ with these conditions are as follows:
		\[
			\begin{array}{l l l}
				\begin{bmatrix} 0 & & & h \\ & \lambda \\ & & -\lambda \\ & & & 0 \end{bmatrix} (\lambda \neq 0) \sim_p
					\begin{bmatrix} 0 & & & h \\ & 1 \\ & & -1 \\ & & & 0 \end{bmatrix} & \text{or} &
				\begin{bmatrix} 0 & & & h \\ & 0 & 1 \\ & -1 & 0 \\ & & & 0 \end{bmatrix}.
			\end{array}
		\]
		We normalize $h = 1$ in the first JCF by scaling $x'_1 = hx_1$ and $x'_3 = hx_3$.
		The same thing also occurs in the second JCF by scaling 
		$x'_1 = hx_1$, $x'_2 = \sqrt{h}x_2$, $x'_3 = \sqrt{h}x_3$ when $h>0$, and changing 
		$x'_1 = hx_1$, $x'_2 = \sqrt{-h}x_3$, $x'_3 = \sqrt{-h}x_2$ when $h < 0$. 
		 Finally, by Proposition \ref{Pro5.4}, we obtain two Lie algebras $L_F$ and 
		 $L_G$ in $\mathrm{Lie}_\ad(5,3)$ with derived algebra $H = \h_3$, where
		\[
			\begin{array}{l l l}
				F \co \begin{bmatrix} 0 & & & 1 \\ & 1 \\ & & -1 \\ & & & 0 \end{bmatrix} & \text{and} &
				G \co \begin{bmatrix} 0 & & & 1 \\ & 0 & 1 \\ & -1 & 0 \\ & & & 0 \end{bmatrix}.
			\end{array}
		\]
	\end{enumerate}
\end{ex}

\begin{rem}
	We summarize intersections between the classification of {\Lntwoa} and 
	Mubarakzyanov \cite{Mub63a,Mub63b} as in Table \ref{tab2}.
	The table suggests that our classification is more compact (cf. also \cite[Part 4]{SW14} for more details).
	\begin{table}[!h]
		\centering
		\begin{tabular}{c c c l c c l c}
			\hline\noalign{\smallskip}
			\rotatebox[origin=c]{90}{Examples} & \rotatebox[origin=c]{90}{$\dim L$} & \rotatebox[origin=c]{90}{$[L, L]$} 
			& \rotatebox[origin=c]{90}{Types} & \rotatebox[origin=c]{90}{Notes} & \rotatebox[origin=c]{90}{\cite[\S 5]{Mub63a}} 
			& \rotatebox[origin=c]{90}{\cite[\S 10]{Mub63b}} & \rotatebox[origin=c]{90}{Notes} \\
			\noalign{\smallskip}\hline\noalign{\smallskip}
			\multirow{2}{*}{\ref{Ex5.6}} & \multirow{2}{*}{4} & \multirow{2}{*}{$\R^2$} & $L_A$ & & $g_{4,1}$ & & \\
			& & & $L_B$ & & $g_{4,3}$ & & \\
			\noalign{\smallskip}\hline\noalign{\smallskip}
			\multirow{7}{*}{\ref{Ex5.7}} &\multirow{7}{*}{5} & \multirow{5}{*}{$\R^3$} & $L_{A(\lambda)}$ & $0 < |\alpha| \leq 1$ 
			& & $g_{5,8}^\gamma$ & $0 < |\gamma| \leq 1$ \\
			& & & $L_B$ & & & $g_{5,15}^{0}$ & \\
			& & & $L_C$ & & & $g_{5,10}$ & \\
			& & & $L_D$ & & & $g_{5,2}$ & \\
			& & & $L_{E(\lambda)}$ & $|\lambda| \geq 0$ & & $g_{5,14}^p$ & \\
			\noalign{\smallskip}\cline{3-8}\noalign{\smallskip}
			& & \multirow{2}{*}{$\h_3$} & $L_F$ & & & $g_{5,20}^{-1}$ & \\
			& & & $L_G$ & & & $g_{5,26}^{\epsilon0}$ & $\epsilon = \pm 1$ \\
			\noalign{\smallskip}\hline
		\end{tabular}
		\caption{{\Lntwoa} and Mubarakzyanov \cite{Mub63a,Mub63b} in low dimensions.}\label{tab2} 
	\end{table}
\end{rem}



\end{document}